\documentclass[11pt,letterpaper,reqno]{amsart}

\usepackage{amssymb, amsmath, amsthm}
\usepackage[colorlinks=true, urlcolor=blue, linkcolor=blue, citecolor=blue]{hyperref}
\usepackage[alphabetic,lite,nobysame]{amsrefs}
\usepackage{verbatim}
\usepackage{amscd}   % for commutative diagrams
\usepackage[all]{xy} % for complicated commutative diagrams
\usepackage{youngtab} % for Young tableaux
\usepackage{young} % for Young tableaux
\usepackage{ytableau}
\usepackage{tikz}
\usepackage{ mathrsfs }
\usepackage{cases}
\usepackage{array}
\usepackage{cellspace}
\usepackage{tabu}
\usepackage{calligra,mathrsfs}
\usepackage{bm}
\usepackage{mathtools}
\usepackage{tikz-cd}

\usepackage[margin=1in]{geometry}

\DeclareMathOperator{\ShHom}{\mathscr{H}\text{\kern -3pt {\calligra\large om}}\,}

\newcommand{\CC}{\mathbb{C}}

\newcommand{\m}{\mathfrak{m}}

\newcommand{\Tor}{\operatorname{Tor}}

\newcommand{\rk}{\operatorname{rank}}

\newcommand{\Sym}{\operatorname{Sym}}

\newcommand{\coker}{\operatorname{coker}}

\renewcommand{\ker}{\operatorname{ker}}

\newcommand{\bb}[1]{\mathbb{#1}}

\newcommand{\mc}[1]{\mathcal{#1}}

\def\PP{{\mathbb P}}
\def\lra{\longrightarrow}

\newtheorem{theorem}{Theorem}[section]
\newtheorem*{theorem*}{Theorem}
\newtheorem*{problem*}{Problem}
\newtheorem{lemma}[theorem]{Lemma}
\newtheorem{conjecture}[theorem]{Conjecture}

\newtheorem*{corollary*}{Corollary}

\theoremstyle{definition}

\newtheorem*{definition*}{Definition}

\theoremstyle{remark}

\newtheorem*{remark*}{Remark}

\numberwithin{equation}{section}

\begin{document}
\title{A proof of the Eisenbud--Huneke--Ulrich conjecture}
\author{Fuxiang Yang}
\address{Department of Mathematics, University of Notre Dame, 255 Hurley, Notre Dame, IN 46556}
\email{fyang6@nd.edu}


\date{\today}


\begin{abstract}
    We prove the Eisenbud--Huneke--Ulrich conjecture on powers of $\mathfrak{m}$-primary ideals with partially linear minimal free resolutions in characteristic zero.
\end{abstract}
\maketitle
\vspace{-.2in}
\section{Introduction}\label{sec:intro}
Let $S = \CC[x_0,\dots,x_n]$, $\m = (x_0,\dots,x_n)$, and $I$ be an ideal of $S$ generated by homogeneous polynomials $f_0,\dots,f_m$ of degree $d$. Hilbert's Nullstellensatz implies that the following are equivalent:
\begin{enumerate}
    \item\label{enu:1} $F\in I$ for every homogeneous polynomial $F\in S$ of sufficiently large degree, that is,
    \begin{equation}\label{eq:mprimary}
        I_t = \m_t \quad \text{for all sufficiently large }t.
    \end{equation}
    \item The linear series $(V,\mc{O}_{\PP^n}(d))$ is basepoint-free, where $V= \operatorname{span}(f_0,\dots,f_m)$.
    \item The common zero locus of $f_0,\dots,f_m$ is supported at the origin.
\end{enumerate}
We say that $I$ is \textbf{$\m$-primary} if these conditions hold. 
Eisenbud, Huneke, and Ulrich proved in \cite{EHU}*{Theorem~1.2} that if $I$ is $\m$-primary and linearly presented, then $(V,\mc{O}_{\PP^n}(d))$ is very ample. This amounts to the stronger condition that every homogeneous polynomial $F$ of sufficiently large degree divisible by $d$ can be written as a linear combination of monomials in $f_0,\dots,f_m$.~Equivalently,
\begin{equation}\label{eq:EHUasymp}
    I^t = \m^{td} \quad \text{for all sufficiently large }t.
\end{equation}
Assume that $I$ is $\m$-primary. While Macaulay established an effective bound for (\ref{eq:mprimary}) \cite{Macaulay}*{\S86}, a corresponding bound for (\ref{eq:EHUasymp}) remained mysterious. Eisenbud, Huenke, and Ulrich conjectured such a bound in terms of the linear syzygies of $I$. We say that the resolution of $I$ is \textbf{linear for $p$ steps} if the linear strand of the minimal free resolution of $I$
\begin{equation}\label{eq:linearStrand}
    \cdots \lra S(-d-p)^{\beta_{p,d}} \lra S(-d-p+1)^{\beta_{p-1,d}} \lra \cdots \lra S(-d)^{\beta_{0,d}} \lra I \lra 0
\end{equation}
is exact through the first $p-1$ steps. In particular, the case $p = 1$ means that $I$ is linearly presented.
\begin{conjecture}[\cite{EHU}*{Conjecture~1.4}]\label{conj:EHU}
    For $p \ge 1$, if $I$ is an $\m$-primary ideal generated in degree $d$ such that its minimal free resolution is linear for $p$ steps, then
    \[I^t = \m^{td} \quad \text{for all }t \ge \left\lceil \frac{n}{p}\right\rceil.\]
\end{conjecture}
\noindent
Eisenbud, Huneke, and Ulrich proved Conjecture~\ref{conj:EHU} for monomial ideals \cite{EHU}*{Theorem~8.1} and for $p \ge n/2$ \cite{EHU}*{Corollary~7.7}.  We previously studied this problem through Castelnuovo--Mumford regularity \cite{yang} and global generation of higher syzygy bundles \cite{yang2}. These methods led to a proof of Conjecture~\ref{conj:EHU} in the case $p = 1$ and gave a non-optimal effective bound for all values of $p$. In this paper, we introduce a surprisingly simple idea based on the theory of Schur complexes to prove Conjecture~\ref{conj:EHU} in a more general setting. We say that the resolution of $I$ is \textbf{virtually linear for $p$ steps} if the sheafification of (\ref{eq:linearStrand}) is exact through the first $p-1$ steps.
\begin{theorem}\label{thm:EHU}
    For $p \ge 1$, if $I$ is an $\m$-primary ideal generated in degree $d$ such that its minimal free resolution is virtually linear for $p$ steps, then
    \[I^t = \m^{td} \quad \text{for all }t \ge \left\lceil \frac{n}{p}\right\rceil.\]
\end{theorem}

\section{Preliminaries}
\subsection{Notation and Conventions}
For a coherent sheaf $\mc{F}$ on $\PP^n$, we denote its $i$-th sheaf cohomology group by $H^i(\mc{F}) = H^i(\PP^n,\mc{F})$. All complexes will be indexed cohomologically. For a complex $\mc{C}^\bullet$ of coherent sheaves on $\PP^n$, the $i$-th hypercohomology is denoted by $\mathbb{H}^i(\mc{C}^\bullet)$.
\subsection{Spectral sequences}
For details on spectral sequences, we refer the reader to \cite{Weibel}*{Section~5.7}. Let $\mc{D}^{\bullet,\bullet}$ be a bounded double complex of coherent sheaves on $\PP^n$. There is a spectral sequence $E$ associated to $\mc{D}^{\bullet,\bullet}$ such that
\[E_1^{k,j} = \bb{H}^j(\mc{D}^{k,\bullet}) \implies \bb{H}^{k+j}(\operatorname{Tot}(\mc{D}^{\bullet,\bullet})) \quad \text{with differential }d_r \colon E_r^{k,j} \lra E_r^{k+r,j-r+1}.\]
\begin{lemma}\label{lem:linearComplex}
    Let $\mc{C}^\bullet$ be a bounded complex of coherent sheaves on $\PP^n$ such that
    \begin{enumerate}
        \item $\mc{C}^\bullet$ is concentrated in nonpositive cohomological degrees,
        \item $\mc{C}^{k}$ is a direct sum of $\mc{O}_{\PP^n}(k)$.
    \end{enumerate}
    Then $\bb{H}^i(\mc{C}^\bullet) = 0$ for all $i \ge 1$.
\end{lemma}
\begin{proof}
    Consider the spectral sequence $E$ on $\mc{C}^\bullet$ by viewing each term $\mc{C}^k$ as a complex concentrated in cohomological degree $0$. To prove $\bb{H}^i(\mc{C}^\bullet) = 0$, it suffices to show that 
    \begin{equation}\label{eq:vanishingE1}
        E_1^{k,i-k} = 0 \quad \text{for all }i \ge 1 \text{ and }k \le 0.
    \end{equation}
    Since $E_1^{k,i-k} = H^{i-k}(\mc{C}^k)$ and $\mc{C}^{k}$ is a direct sum of $\mc{O}_{\PP^n}(k)$, the desired vanishing (\ref{eq:vanishingE1}) follows from \cite{Hartshorne}*{Theorem~III.5.1}.
\end{proof}
\subsection{Schur complexes} For details on Schur complexes, we refer the reader to \cite{weyman}*{Chapter~2}.  Let
\[0 \lra \mc{B}^\bullet \lra \mc{A}^\bullet \lra L \lra 0\]
be a short exact sequence of bounded complexes of vector bundles on $\PP^n$, where $L$ is a line bundle concentrated in cohomological degree $0$. Since $L$ has rank $1$, for every positive integer $r$, there are two natural short exact sequences given by
\begin{equation}\label{eq:symses}
    0 \lra \Sym^r(\mc{B}^\bullet) \lra \Sym^r(\mc{A}^\bullet) \lra \Sym^{r-1}(\mc{A}^\bullet) \otimes L \lra 0,
\end{equation}
\begin{equation}\label{eq:wedgeses}
    0 \lra \bigwedge^r(\mc{B}^\bullet) \lra \bigwedge^r(\mc{A}^\bullet) \lra \bigwedge^{r-1}(\mc{B}^\bullet) \otimes L \lra 0.
\end{equation}
In fact, this construction can be generalized to all hook Schur functors. We assume that $L$ is placed at cohomological degree $1$.
\begin{lemma}\label{lem:SEShooks}
    For all integer $i,r$ such that $0 \le i \le r-1$, there is a short exact sequence
    \[0 \lra \bb{S}_{(r-i,1^i)}(\mc{B}^\bullet) \lra \bb{S}_{(r-i,1^i)}(\mc{A}^\bullet) \lra \Sym^{r-i-1}(\mc{A}^\bullet) \otimes \bigwedge^{i}(\mc{B}^\bullet) \otimes L \lra 0.\]
    In particular, when $i = 0$, this recovers (\ref{eq:symses}), and when $i = r-1$, this recovers (\ref{eq:wedgeses}).
\end{lemma}
\begin{proof}
    Write $\lambda = (r-i,1^i)$. For any partition $\mu$, we denote its tranpose by $\mu^t$. Apply the functor $\bb{S}_{\lambda}(-)$ to the surjection $\mc{A}^\bullet \lra L$, by \cite{weyman}*{Theorem~2.4.10}, we get a double complex $\mc{C}^{\bullet,\bullet}$ with
    \[\mc{C}^{k,\bullet} = \bigoplus_{\substack{\mu \subseteq \lambda\\|\mu| = k}} \bb{S}_{\lambda/\mu}(\mc{A}^\bullet) \otimes \bb{S}_{\mu^t}(L) \overset{\rk(L)=1}{=} \begin{cases}
        \bb{S}_{\lambda}(\mc{A}^\bullet) &k=0,\\
        \Sym^{r-i-1}(\mc{A}^\bullet) \otimes \bigwedge^{i+1-k}(\mc{A}^\bullet) \otimes L^{k} &1 \le k \le i+1.
    \end{cases}\]
    Each of the complex $\mc{C}^{\bullet,j}$ only has nonzero cohomology at degree $0$, and it is given by $(\bb{S}_\lambda(\mc{B}^\bullet))^j$. Hence, it suffices to show that
    \[\Sym^{r-i-1}(\mc{A}^\bullet) \otimes \bigwedge^{i}(\mc{B}^\bullet) \otimes L = \ker(\Sym^{r-i-1}(\mc{A}^\bullet) \otimes \bigwedge^{i}(\mc{A}^\bullet) \otimes L \overset{\partial}{\lra} \Sym^{r-i-1}(\mc{A}^\bullet) \otimes \bigwedge^{i-1}(\mc{A}^\bullet) \otimes L^2).\]
    Indeed, the map $\partial$ is the tensor product of $\Sym^{r-i-1}(\mc{A}^\bullet) \otimes L$ with the usual Koszul differential
    \[\bigwedge^i (\mc{A}^\bullet) \lra \bigwedge^{i-1}(\mc{A}^\bullet) \otimes L,\]
    which has kernel $\bigwedge^i (\mc{B}^\bullet)$. This concludes the proof.
\end{proof}

\section{Proof of the Eisenbud--Huneke--Ulrich conjecture}
Let $I$ be an $\m$-primary ideal generated in degree $d$. Assume the resolution of $I$ is virtually linear for $p$ steps. Write
\[V_i = \Tor_i^S(I,\CC)_{i+d} \quad \text{for all }0 \le i \le p.\]
It follows from the virtual linearity that there is a complex
\[\mc{A}^\bullet \colon 0 \lra V_p(-p) \lra \cdots \lra V_1(-1) \lra V_0 \otimes \mc{O}_{\PP^n}\lra 0\]
such that the cohomology sheaves are given by
\[\mc{H}^{i}(\mc{A}^\bullet)=\begin{cases}
    \mc{O}_{\PP^n}(d) &i = 0,\\
    E_p &i=-p,\\
    0 &\text{otherwise,}
\end{cases}\]
for some vector bundle $E_p$. Set $\mc{B}^\bullet \coloneqq \ker(\mc{A}^\bullet \lra \mc{O}_{\PP^n}(d))$. The cohomology sheaves of $\mc{B}^\bullet$ are
\[\mc{H}^{i}(\mc{B}^\bullet)=\begin{cases}
    E_p &i=-p,\\
    0 &\text{otherwise.}
\end{cases}\]
In particular, we have a short exact sequence
\begin{equation}\label{eq:BALses}
    0 \lra \mc{B}^\bullet \lra \mc{A}^\bullet \lra \mc{O}_{\PP^n}(d) \lra 0.
\end{equation}
\begin{lemma}\label{lem:hypercohomologyVanishingQuotientSchur}
    For all $t\ge \left\lceil \frac{n}{p}\right\rceil$ and $0 \le i \le t-1$, we have
    \[\bb{H}^j\left( \Sym^{t-i-1}(\mc{A}^\bullet) \otimes \bigwedge^i(\mc{B}^\bullet) \otimes \mc{O}_{\PP^n}(d)\right) = 0 \quad \text{for all }j \ge 1.\]
\end{lemma}
\begin{proof}
    Apply Lemma~\ref{lem:SEShooks} to (\ref{eq:BALses}), there is a short exact sequence
    \[0 \lra \bb{S}_{(t-i,1^i)}(\mc{B}^\bullet) \lra \bb{S}_{(t-i,1^i)}(\mc{A}^\bullet) \lra \Sym^{t-i-1}(\mc{A}^\bullet) \otimes \bigwedge^{i}(\mc{B}^\bullet) \otimes \mc{O}_{\PP^n}(d) \lra 0.\]
    It follows from the induced long exact sequence in hypercohomology that it suffices to show 
    \[\bb{H}^j(\bb{S}_{(t-i,1^i)}(\mc{A}^\bullet)) = \bb{H}^{j+1}(\bb{S}_{(t-i,1^i)}(\mc{B}^\bullet)) = 0\quad \text{for all }j \ge 1.\]
    Note that $\bb{S}_{(t-i,1^i)}(\mc{A}^\bullet)$ is a complex concentrated in nonpositive cohomological degrees, where the term $(\bb{S}_{(t-i,1^i)}(\mc{A}^\bullet))^k$ is a direct sum of $\mc{O}_{\PP^n}(k)$ for all $k \le 0$. By Lemma~\ref{lem:linearComplex}, $\bb{H}^j(\bb{S}_{(t-i,1^i)}(\mc{A}^\bullet)) = 0$ for all $j \ge 1$. Since the complex $\bb{S}_{(t-i,1^i)}(\mc{B}^\bullet)$ is a direct summand of $(\mc{B}^\bullet)^{\otimes t}$, there is an injection
    \[\bb{H}^{j+1}(\bb{S}_{(t-i,1^i)}(\mc{B}^\bullet)) \lra \bb{H}^{j+1}((\mc{B}^\bullet)^{\otimes t}).\]
    Since $\mc{B}^\bullet$ has nonzero cohomology sheaf only in degree $-p$, $(\mc{B}^\bullet)^{\otimes t}$ has nonzero cohomology sheaf only in degree $-tp$, and it is given by $E_p^{\otimes t}$. Hence, the hypercohomology of the complex $(\mc{B}^\bullet)^{\otimes t}$ can be computed by
    \begin{equation}\label{eq:hypercohomologyTensorB}
        \bb{H}^{j+1}((\mc{B}^\bullet)^{\otimes t}) = H^{tp+j+1}(E_p^{\otimes t}).
    \end{equation}
    It follows from the Grothendieck vanishing and the ineqality $tp+j+1 > n$ that (\ref{eq:hypercohomologyTensorB}) vanishes. Hence, $\bb{H}^{j+1}(\bb{S}_{(t-i,1^i)}(\mc{B}^\bullet)) = 0$ for all $j \ge 0$. This concludes the proof.
\end{proof}
The $(r-1)$-st symmetric product of the natural inclusion $\mc{B}^\bullet \lra \mc{A}^\bullet$ gives an acyclic complex with cohomology $\mc{O}_{\PP^n}((r-1)d)$. Taking the associated mapping cone and tensoring with $\mc{O}_{\PP^n}(d)$, we get an exact complex
\begin{align*}
    0 \lra \bigwedge^{r-1}(\mc{B}^\bullet) \otimes \mc{O}_{\PP^n}(d) \lra \cdots \lra \Sym^{r-1-i}(\mc{A}^\bullet) \otimes \bigwedge^i(\mc{B}^\bullet) \otimes \mc{O}_{\PP^n}(d) \lra \cdots \lra\\
    \Sym^{r-2}(\mc{A}^\bullet) \otimes \mc{B}^{\bullet}\otimes \mc{O}_{\PP^n}(d) \lra \Sym^{r-1}(\mc{A}^\bullet)\otimes \mc{O}_{\PP^n}(d) \lra \mc{O}_{\PP^n}(rd)\lra 0.
\end{align*}
We denote this by $\mc{D}^{\bullet,\bullet}_r$, where we have
\[\mc{D}_r^{0,\bullet} = \mc{O}_{\PP^n}(rd)\quad \text{and} \quad \mc{D}_r^{-i,\bullet} = \Sym^{r-i}(\mc{A}^\bullet) \otimes \bigwedge^{i-1}(\mc{B}^\bullet) \otimes \mc{O}_{\PP^n}(d) \quad \text{for all }1 \le i \le r.\]
\begin{proof}[Proof of Theorem~\ref{thm:EHU}]
    We show that $(S/I^t)_{td} = 0$ for all $t\ge \left\lceil \frac{n}{p}\right\rceil$. We first prove
    \begin{equation}\label{eq:hilbertfunctioncoker}
        (S/I^t)_{td} = \coker(\bb{H}^0(\mc{D}^{-1,\bullet}) \lra \bb{H}^0(\mc{D}^{0,\bullet})).
    \end{equation}
    By Lemma~\ref{lem:hypercohomologyVanishingQuotientSchur}, we have $\bb{H}^1(\mc{D}^{-1,\bullet}) = \bb{H}^1(\Sym^{t-1}(\mc{A}^\bullet) \otimes \mc{O}_{\PP^n}(d)) = 0$ which implies that
    \begin{align*}
        \coker(\bb{H}^0(\mc{D}^{-1,\bullet}) \lra \bb{H}^0(\mc{D}^{0,\bullet})) &= \ker(\bb{H}^0(\operatorname{cone}(\mc{D}^{-1,\bullet} \lra \mc{D}^{0,\bullet})) \lra \bb{H}^1(\mc{D}^{-1,\bullet}))\\
        &= \bb{H}^0(\operatorname{cone}(\mc{D}^{-1,\bullet} \lra \mc{D}^{0,\bullet}))\\
        &= \bb{H}^0(\operatorname{cone}(\Sym^{t-1}(\mc{A}^\bullet) \otimes \mc{O}_{\PP^n}(d) \lra \mc{O}_{\PP^n}(td))).
    \end{align*}
    It then follows that 
    \begin{align*}
        (S/I^t)_{td}&\overset{\text{\cite{yang}*{Lemma~2.1}}}{=}\bb{H}^0(\operatorname{cone}(\Sym^{t}(\mc{A}^\bullet) \lra \mc{O}_{\PP^n}(td)))\\
        &\overset{\text{\cite{yang}*{Lemma~4.1}}}{=} \bb{H}^0(\operatorname{cone}(\Sym^{t-1}(\mc{A}^\bullet) \otimes \mc{O}_{\PP^n}(d) \lra \mc{O}_{\PP^n}(td)))\\
        &\overset{\hphantom{\text{\cite{yang}*{Lemma~4.1}}}}{=} \coker(\bb{H}^0(\mc{D}^{-1,\bullet}) \lra \bb{H}^0(\mc{D}^{0,\bullet})).
    \end{align*}
    We claim that
    \begin{equation}\label{eq:tot=coker}
        \bb{H}^0(\operatorname{Tot}(\mc{D}^{\bullet,\bullet}_t)) = \coker(\bb{H}^0(\mc{D}^{-1,\bullet}) \lra \bb{H}^0(\mc{D}^{0,\bullet})).
    \end{equation}
    Since $\operatorname{Tot}(\mc{D}^{\bullet,\bullet}_t)$ is exact and $\bb{H}^0(\operatorname{Tot}(\mc{D}^{\bullet,\bullet}_t)) = 0$, the desired vanishing $(S/I^t)_{td} = 0$ follows from (\ref{eq:tot=coker}). We now prove the claim (\ref{eq:tot=coker}).
    Consider the spectral sequence $E$ associated to $\mc{D}^{\bullet,\bullet}_t$. By Lemma~\ref{lem:hypercohomologyVanishingQuotientSchur},
    \begin{align*}
        E_1^{-i-1,i} = \bb{H}^i(\mc{D}^{-i-1,\bullet}) = \bb{H}^i\left( \Sym^{t-i-1}(\mc{A}^\bullet) \otimes \bigwedge^i(\mc{B}^\bullet) \otimes \mc{O}_{\PP^n}(d)\right) &= 0 \quad \text{for all }1 \le i \le t-1,\\
        E_1^{-i,i} = \bb{H}^i(\mc{D}^{-i,\bullet}) = \bb{H}^i\left( \Sym^{t-i}(\mc{A}^\bullet) \otimes \bigwedge^{i-1}(\mc{B}^\bullet) \otimes \mc{O}_{\PP^n}(d)\right) &= 0 \quad \text{for all } 1 \le i \le t.
    \end{align*}
    It follows that the only potentially nonzero differential mapping into $E_1^{0,0}$ is 
    \[d_1 \colon \bb{H}^0(\mc{D}^{-1,\bullet}) =  E_1^{-1,0} \lra E_1^{0,0} = \bb{H}^0(\mc{D}^{0,\bullet})),\]
    and that $\coker(d_1) = E_\infty^{0,0} = \bb{H}^0(\operatorname{Tot}(\mc{D}^{\bullet,\bullet}_t))$. 
    This concludes the proof.
\end{proof}

\subsection*{Acknowledgment} I would like to thank Claudiu Raicu for his guidance, support, and valuable suggestions throughout
this project. The author acknowledges support from the Simons Dissertation Fellowship SFI-MPS-SDF-00023235, the Arthur J. Schmitt Fellowship.

\subsection*{AI disclosure} Some preliminary ideas related to Lemma~\ref{lem:SEShooks} were developed by the author in part through a conversation with ChatGPT.

\end{document}